\documentclass[12pt]{amsart}
\usepackage[utf8]{inputenc}
\usepackage[T1]{fontenc}
\usepackage{amsfonts}
\usepackage[leqno]{amsmath}
\usepackage{amssymb, comment, float}
\usepackage[dvipsnames]{xcolor}
\usepackage[foot]{amsaddr}
\usepackage[shortlabels]{enumitem}
\usepackage{mathdots}
\usepackage[a4paper, footskip=0.5in,  headheight = 0.5in, top=1.25in, bottom=1.25in,  right=1in,  left=1in]{geometry}
\usepackage{hyperref}
\hypersetup{
colorlinks,
linkcolor=black,
urlcolor=Blue,
citecolor=black,}
\usepackage{cleveref}
\usepackage[center]{caption}
\usepackage{eufrak}
\usepackage{marvosym}     
\usepackage{latexsym}
\usepackage[nodayofweek]{datetime}
\usepackage{tikz}
\usepackage{subfig}
\usepackage{thm-restate}

\newtheorem{theorem}{Theorem}[section]
\newtheorem{lemma}[theorem]{Lemma}
\newtheorem{problem}[theorem]{Problem}

\newtheorem{corollary}[theorem]{Corollary}

\newtheorem{observation}[theorem]{Observation}

\def\dd{\hbox{-}}   

\newcommand{\mca}{\mathcal}
\newcommand{\poi}{\mathbb{N}}

\newcounter{tbox}
\newcommand{\sta}[1]{\medskip\medskip\refstepcounter{tbox}\noindent{\parbox{\textwidth}{(\thetbox) \emph{#1}}}\vspace*{0.3cm}}
\newcommand{\mylongtitle}[1]{%
  \ifodd\value{page}%
    \protect\parbox{0.97\linewidth}{#1}\hfill%
  \else%
    \hfill\protect\parbox{0.97\linewidth}{#1}%
  \fi%
}

\title[Suns in triangle-free graphs of large chromatic number]{Suns in triangle-free graphs of large chromatic number}

\author{Sepehr Hajebi$^{\dagger}$}
\author{Sophie Spirkl$^{\dagger \ast}$}

\thanks{$^{\dagger}$ Department of Combinatorics and Optimization, University of Waterloo, Waterloo, Ontario, Canada}
\thanks{$^{\ast}$ We acknowledge the support of the Natural Sciences and Engineering Research Council of Canada (NSERC), [funding reference number RGPIN-2020-03912].
Cette recherche a \'et\'e financ\'ee par le Conseil de recherches en sciences naturelles et en g\'enie du Canada (CRSNG), [num\'ero de r\'ef\'erence RGPIN-2020-03912]. This project was funded in part by the Government of Ontario. This research was conducted while Spirkl was an Alfred P. Sloan Fellow.}
\date{\today}
\begin{document}
\maketitle

\begin{abstract}

For an integer $t\geq 4$, a \textit{$t$-sun} is a graph obtained from a $t$-vertex cycle $C$ by adding a degree-one neighbor for each vertex of $C$. Trotignon asked whether every triangle-free graph of sufficiently large chromatic number has an induced subgraph that is a $t$-sun for some $t\geq 4$. This remains open, but we show that every triangle-free graph of chromatic number at least $48$ has an induced subgraph that is either a $t$-sun for some $t\geq 5$, or a $4$-sun with a single degree-one vertex deleted. In fact, we prove that for all $\ell\geq 5$, there exists $c=c(\ell)\in \poi$ such that every triangle-free graph of chromatic number at least $c$ has an induced subgraph that is either a $t$-sun for some $t\geq 
\ell$, or a $4$-sun with a single degree-one vertex deleted.
\end{abstract}
 
\section{Introduction}

The set of all positive integers is denoted by $\poi$, and for integers $k,k'$, we denote by $\{k,\ldots, k'\}$ the set of all integers no smaller than $k$ and no larger than $k'$ (thus, $\{k,\ldots, k'\}=\varnothing$ if and only if $k'<k$). Graphs in this paper have finite vertex sets, no loops, and no parallel edges. We use the same notation for induced subgraph and their vertex sets; in particular, for an induced subgraph $H$ of a graph $G$, we also use $H$ to denote the vertex set of $H$. The chromatic number and the clique number of a graph $G$ are denoted, in order, by $\chi(G)$ and $\omega(G)$. Recall that a graph $G$ is \textit{triangle-free} if $\omega(G)\leq 2$.

Let $t\geq 4$ be an integer. A \textit{$t$-sun} is the graph obtained from a $t$-vertex cycle by adding a degree-one neighbor for each vertex of the cycle. A graph $G$ is \textit{$t$-sun-free} if no induced subgraph of $G$ is a $t$-sun, and $G$ is \textit{sun-free} if $G$ is $t$-sun-free for all $t\geq 4$. Recently, Trotignon \cite{barbados2025} proposed the following beautiful problem:

\begin{problem}[Trotignon \cite{barbados2025}]\label{prb:trotignon}
Is $\chi(G)$ bounded for every triangle-free sun-free graph $G$?
\end{problem}

This remains open, but we come pretty close to a solution. Given an integer $t\geq 4$, a \textit{$t$-sunspot} is the graph obtained from a $t$-sun by removing a single degree-one vertex. A graph $G$ is \textit{$t$-sunspot-free} if no induced subgraph of $G$ is a $t$-sunspot (so, being $t$-sunspot-free implies being $t$-sun-free). We prove the following (see Figure~\ref{fig:suns}):

\begin{theorem}\label{thm:maintrianglefreeshort}
 Let $G$ be a graph that is triangle-free, $4$-sunspot-free, and $t$-sun-free for all $t\geq 5$. Then $\chi(G)\leq 47$.  
\end{theorem}
\begin{figure}[t!]
    \centering
    \includegraphics[scale=0.8]{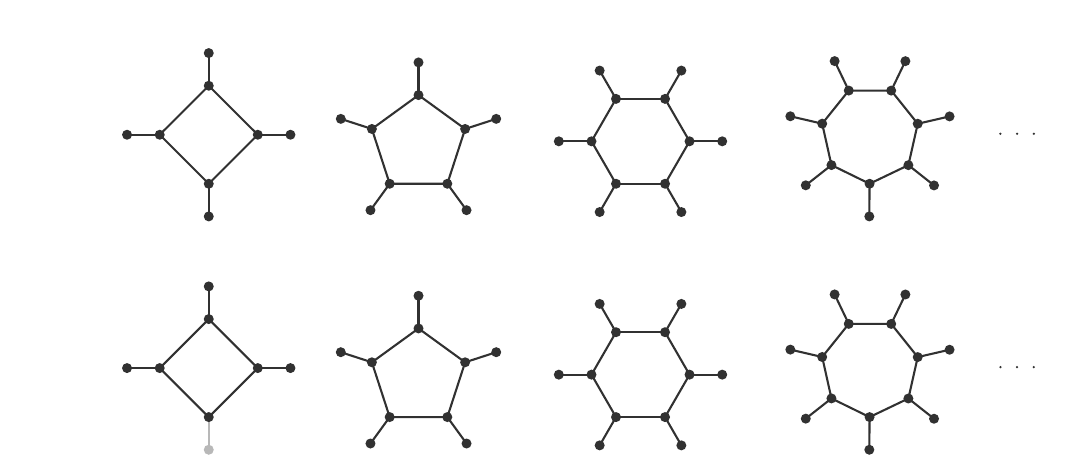}
    \caption{Suns (top) versus what we exclude in \Cref{thm:maintrianglefreeshort} (bottom).}
    \label{fig:suns}
\end{figure}

In particular, we will only work with $4$-sunspots, so let us customize the notation and terminology in this specific case. Given a graph $G$, by a \textit{$4$-sunspot in $G$} we mean a $7$-tuple $(x_1,x_2,x_3,x_4;y_1,y_2,y_3)$ of pairwise distinct vertices of $G$ such that 
$$E(G[\{x_1,x_2,x_3,x_4,y_1,y_2,y_3\}])=\{x_1x_2,x_2x_3,x_3x_4,x_4x_1, x_1y_1,x_2y_2,x_3y_3\}.$$
It follows that $G$ is $4$-sunspot-free if and only if there is no $4$-sunspot in $G$.

As pointed out in \cite{barbados2025}, it is not difficult check that ``shift graphs'' (see \cite{chisurvey}) provide a construction of triangle-free graphs with arbitrarily large chromatic number and no induced subgraph that is a $t$-sun for any $t\geq 5$ (we omit the details). This means that excluding $4$-suns is essential for the answer to \Cref{prb:trotignon} to be affirmative. However, provided the $4$-suns are excluded, it could be that the answer to \Cref{prb:trotignon} is ``yes'' even if in addition we only exclude $t$-suns for all $t\geq \ell$, where $\ell\geq 5$ is a prescribed constant. At least, the corresponding extension of \Cref{thm:maintrianglefreeshort} is true, and that is the main result of this paper:

\begin{restatable}{theorem}{maintrianglefree}\label{thm:maintrianglefree}
For every integer $\ell\geq 6$, there is a constant $c_{\ref{thm:maintrianglefree}}=c_{\ref{thm:maintrianglefree}}(\ell)\in \poi$ such that $\chi(G)\leq c_{\ref{thm:maintrianglefree}}$ for every graph $G$ that is triangle-free, $4$-sunspot-free, and $t$-sun-free for all $t\geq \ell$. Moreover, $c_{\ref{thm:maintrianglefree}}(6)=47$ works. 
\end{restatable}

The proof of \Cref{thm:maintrianglefree} will be completed in the last section. 
\medskip

We conclude the introduction by discussing another problem of Trotignon \cite{barbados2025} about extending \Cref{prb:trotignon} to graphs of bounded clique number (for a general bound, instead of just triangle-free graphs).  A graph class $\mca{C}$ is \textit{$\chi$-bounded} if there is a function $f:\poi\rightarrow \poi$ such that $\chi(G)\leq f(\omega(G))$ for every graph $G\in \mca{C}$. A \textit{net} is the graph obtained from a $3$-vertex cycle by adding a degree-one neighbor for each vertex of the cycle, and a graph $G$ is \textit{net-free} if no induced subgraph of $G$ is a net.

\begin{problem}[Trotignon \cite{barbados2025}]\label{prb:trotignon2}
Is the class of all net-free sun-free graphs $\chi$-bounded?
\end{problem}

(A remark: In \cite{barbados2025}, a net is called a ``$3$-sun,'' and a net-free sun-free graph is said to be ``sun-free.'' But we prefer our terminology, especially because the name net is pretty standard. Another reason is that we are not sure why the net needs to be excluded in \Cref{prb:trotignon2}; it could be true that -- again, in our terminology -- the class of all sun-free graphs is already $\chi$-bounded.)

It turns out that a partial answer to \Cref{prb:trotignon2}, analogous to \Cref{thm:maintrianglefree}, can in fact be derived from \Cref{thm:maintrianglefree}. A \textit{bull} is the graph obtained from a net by removing a degree-one vertex, and a graph $G$ is \textit{bull-free} if no induced subgraph of $G$ is a bull. The following was recently proved in \cite{bullfree} (the logarithm is base 2):

\begin{theorem}[Hajebi \cite{bullfree}]\label{thm:bullfree}
    Let $t\in \poi$ and let $G$ be a bull-free graph in which every triangle-free induced subgraph has chromatic number at most $t$. Then we have $\chi(G)\leq \omega(G)^{4\log t+13}$.
\end{theorem}

Since $\log 43<6$, from \Cref{thm:maintrianglefree,thm:bullfree}, it follows that:

\begin{theorem}\label{thm:mainlong}
Let $G$ be a graph that is bull-free, $4$-sunspot-free, and $t$-sun-free for all $t\geq 5$. Then $\chi(G)\leq \omega(G)^{37}$. 
\end{theorem}

In fact, for every integer $\ell\geq 6$, there exists $d=d(\ell)\in \poi$ such that $\chi(G)\leq \omega(G)^d$ for every graph $G$ that is bull-free, $4$-sunspot-free, and $t$-sun-free for all $t\geq \ell$.

\section{Flaps}\label{sec:flap}

Let $G$ be a graph. For $x\in V(G)$ and $Y\subseteq V(G)$, we denote by $N_Y(x)$ the set of all vertices in $Y$ that are adjacent to $x$ in $G$. A \textit{hole in $G$} is an induced subgraph $H$ of $G$ that is a cycle on four or more vertices. The \textit{length} of $H$ is the number of edges of $H$; if $H$ has length $\ell\geq 4$, then we also say that $H$ is an \textit{$\ell$-hole in $G$}, and we write $H=h_1\dd \cdots h_{\ell}\dd h_1$ to mean $V(H)=\{h_1,\ldots, h_{\ell}\}$ and $E(H)=\{h_{i}h_{i+1}:i\in \{1,\ldots, \ell-1\}\}\cup \{h_{\ell}h_1\}$. An \textit{$H$-flap in $G$} is a $4$-hole in $G$ that shares at least one edge with $H$. We say that a graph $G$ is
\begin{itemize}
    \item \textit{non-degenerate} if every vertex in $G$ has degree at least $\chi(G)-1$;
    \item \textit{liberal} for all distinct and non-adjacent $x,y\in V(G)$, neither $x$ nor $y$ ``dominates'' the other; that is, we have $N_{G}(x)\setminus N_{G}(y)\neq \varnothing$ and $N_{G}(y)\setminus N_{G}(x)\neq \varnothing$; and
    \item \textit{flapless} if for every hole $H$ of length at least $6$ in $G$, there is no $H$-flap in $G$. 
\end{itemize}

Our goal in this section is to prove the following:

\begin{theorem}\label{thm:flap}
    Let $c\in \poi$ and let $G$ be a triangle-free $4$-sunspot-free graph such that $\chi(G)> 2c+1$. Then $G$ has an induced subgraph $L$ with $\chi(L)=c+1$ that is non-degenerate, liberal, and flapless.
\end{theorem}

The proof is in several steps, starting with an observation (we omit the proof):

\begin{observation}\label{obs:critical}
Let $c\in \poi$ and let $G$ be a graph such that $\chi(G)>c$ and $\chi(G')\leq c$ for every induced subgraph $G'$ of $G$ other than $G$ itself. Then $\chi(G)=c+1$, and $G$ is both non-degenerate and liberal.
\end{observation}

Let $r\in \poi\cup \{0\}$ and let $G$ be a graph. An \textit{$r$-leveling in $G$} is an $(r+1)$-tuple $(L_0,\ldots, L_r)$ of pairwise disjoint non-empty subsets of $V(G)$ such that:
\begin{itemize}
    \item for every $i\in \{1,\ldots, r\}$, every vertex in $L_i$ has a neighbor in $L_{i-1}$;
    \item for all distinct $i,j\in \{0,\ldots, r\}$, if there is an edge in $G$ with an end in $L_i$ and an end in $L_j$, then $|i-j|=1$.
\end{itemize}

\begin{lemma}\label{lem:deepleveling}
  Let $c\in \poi$ and let $G$ be a triangle-free graph with $\chi(G)>2c+1$. Then there is an $r$-leveling $(L_0,\ldots, L_r)$ in $G$ for some integer $r\geq 3$ such that $\chi(L_r)=c+1$ and $L_r$ is both non-degenerate and liberal.
\end{lemma}
\begin{proof}
 Let $K$ be a connected component of $G$ with $\chi(K)=\chi(G)> 2c+1$. Let $x_0\in K$ and let $K_1$ be a component of $G\setminus (N_K(x_0)\cup \{x_0\})$ with maximum chromatic number. Since $\chi(K)>2c+1$ and $N_K(x_0)$ is a stable set in $G$ (because $G$ is triangle-free), it follows that $\chi(K_1)> 2c$. Since $K$ is connected, it follows that there is a vertex $x_1\in N_{K}(x_0)$ such that $K_1\cup \{x_1\}$ is connected. For each $r\in \poi$, let $M_r$ be the set of all vertices of $K_1$ that are at distance exactly $r$ from $x_1$ in the graph $K_1\cup \{x_1\}$. Note that $M_1=N_{K_1}(x_1)$ (and $M_r=\varnothing$ for sufficiently large $r$). It follows that $(M_1,\ldots, M_r)$ is an $(r-1)$-leveling in $K_1$ (and so in $G$) for every $r\in \poi$. Also, since $K_1\cup \{x_1\}$ is connected, $K_1=\bigcup_{r\in \poi}M_r$ and by the second property in the definition of a leveling, $\chi(K_1)\leq 2\max_{r\in \poi}\chi(M_r)$. Note that $\chi(M_1)=1$ (because $G$ is triangle-free). We deduce that there exists $r_1\in \poi$ with $r_1\geq 2$ such that $\chi(M_{r_1})\geq \chi(K_1)/2>c\geq 1$. Let $K_2$ be an induced subgraph of $M_{r_1}$ with $V(K_2)$ minimal such that $\chi(K_2)>c$. By \Cref{obs:critical}, we have:

    \sta{\label{st:non-degenerateandliberal}$\chi(K_2)=c+1$ and $K_2$ is both non-degenerate and liberal.}

Now, for each $i\in \{0,\ldots, r_1+1\}$, define $L_i$ as follows: let $L_0=\{x_0\}$, let $L_1=\{x_1\}$, $L_{i}=M_{i-1}$ for every $i\in \{2,\ldots, r_1\}$, and let $L_{r_1+1}=K_2\subseteq M_{r_1}$.
Then $(L_0,\ldots,L_{r_1+1})$ is an $(r_1+1)$-leveling; in particular, we have $r_1+1\geq 3$ since $r_1\geq 2$. Furthermore, by \eqref{st:non-degenerateandliberal}, we have $\chi(L_{r_1+1})=\chi(K_2)=c+1$, and $L_{r_1+1}=K_2$ is both non-degenerate and liberal. This completes the proof of \Cref{lem:deepleveling}.
\end{proof}

Let $G$ be a graph. A \textit{path in $G$} is an induced subgraph $P$ of $G$ that is a path. We write $P=p_1\dd \cdots\dd p_{\ell}$, for $\ell\in \poi$, to mean $V(P)=\{p_1,\ldots, p_{\ell}\}$ and $E(P)=\{p_ip_{i+1}:i\in \{1,\ldots, \ell-1\}\}$. We call $p_1,p_{\ell}$ the \textit{ends} of $P$, and the set $P\setminus \{p_1,p_{\ell}\}$ the \textit{interior} of $P$. The \textit{length} of $P$ is the number of edges of $P$. Given a hole $H$ in $G$ and a vertex $x\in G\setminus H$, an \textit{$x$-sector in $H$} is a path $P$ of non-zero length in $H$ such that $x$ is adjacent in $G$ to the ends of $P$ and $x$ is not adjacent in $G$ to any vertex in the interior of $P$.
\medskip

We need the next three lemmas:

\begin{lemma}\label{lem:fromtop}
    Let $r\geq 2$ be an integer, let $G$ be a triangle-free $4$-sunspot-free graph, and let $(L_0,\ldots, L_r)$ be an $r$-leveling in $G$. Let $H$ be a hole of length at least $6$ in $L_r$ and let $x\in L_{r-1}$. Then there is no $x$-sector of length at most $2$ in $H$.
\end{lemma}
\begin{proof}
There is no $x$-sector of length $1$ in $H$ because $G$ is triangle-free. Assume that there is an $x$-sector of length $2$ in $H$. Since $H$ has length at least $6$ and since $G$ is triangle-free, there is a path $P=h_1\dd h_2\dd h_3\dd h_4\dd h_5$ of length $4$ in $H$ such that $N_P(x)=\{h_2,h_4\}$. Since $r\geq 2$, it follows that $x$ has a neighbor $y\in L_{r-2}$. But now $(h_2,x,h_4,h_3; h_1,y,h_5)$ is a $4$-sunspot in $G$, a contradiction. This completes the proof of \Cref{lem:fromtop}.
\end{proof}

\begin{lemma}\label{lem:frombottom}
    Let $r\geq 3$ be an integer, let $G$ be a triangle-free $4$-sunspot-free graph, and let $(L_0,\ldots, L_r)$ be an $r$-leveling in $G$ such that $L_r$ is liberal. Let $H$ be a hole of length at least $6$ in $L_r$ and let $x\in L_{r}\setminus H$. Then there is no $x$-sector of length at most $2$ in $H$.
\end{lemma}
\begin{proof}
There is no $x$-sector of length $1$ in $H$ because $G$ is triangle-free. Assume that there is an $x$-sector of length $2$ in $H$. Since $H$ has length at least $6$ and since $G$ is triangle-free, there is a path $P=h_1\dd h_2\dd h_3\dd h_4\dd h_5$ of length $4$ in $H$ such that $N_P(x)=\{h_2,h_4\}$. Since $r\geq 3$, it follows that $h_3$ has a neighbor $y\in L_{r-1}$ and $y$ has a neighbor $z\in L_{r-2}$. By \Cref{lem:fromtop}, there is no $y$-sector of length at most $2$ in $H$, and so $N_P(y)=\{h_3\}$. Since $(h_2,h_3,h_4,x; h_1,y,h_5)$ is not a $4$-sunspot in $G$, it follows that $xy\in E(G)$. Since $L_r$ is liberal, there is a vertex $w\in N_{L_r}(x)\setminus N_{L_r}(h_3)$. In particular, we have $w\in L_r\setminus P$. Since $G$ is triangle-free, it follows that $yw\notin E(G)$, and $N_P(w)\subseteq \{h_1,h_5\}$. Also, if $h_1w\notin E(G)$, then $(h_2,x,y,h_3;h_1,w,z)$ is a $4$-sunspot in $G$, and if $h_5w\notin E(G)$, then $(h_4,x,y,h_3;h_5,w,z)$ is a $4$-sunspot in $G$, a contradiction. We deduce that $N_P(w)= \{h_1,h_5\}$, and so $H'=w\dd h_1\dd h_2\dd h_3\dd h_4\dd h_5\dd w$ is a hole of length $6$ in $L_r$. Since $r\geq 3$, it follows that $h_1$ has a neighbor $y_1\in L_{r-1}$, $y_1$ has a neighbor $y_2\in L_{r-2}$, and $y_2$ has a neighbor $y_3\in L_{r-3}$. By \Cref{lem:fromtop} applied to $H'$ and $y_1$, we have $h_1\in N_{H'}(y_1)\subseteq \{h_1,h_4\}$. Consequently, since $(h_2,h_1,w,x;h_3,y_1,h_5)$ is not a $4$-sunspot in $G$, it follows that $xy_1\in E(G)$, and so $yy_1\notin E(G)$ (because $G$ is triangle-free). Now, if $yy_2\in E(G)$, then $(y_1,y_2,y,x;h_1,y_3,h_3)$ is a $4$-sunspot in $G$, a contradiction. It follows that $yy_2\notin E(G)$. But then $(w,x,y_1,h_1;h_5,y,y_2)$ is a $4$-sunspot in $G$, again a contradiction. This completes the proof of \Cref{lem:frombottom}.
\end{proof}

\begin{lemma}\label{lem:topflap}
    Let $r\geq 3$ be an integer, let $G$ be a triangle-free $4$-sunspot-free graph, and let $(L_0,\ldots, L_r)$ be an $r$-leveling in $G$ such that $L_r$ is liberal. Let $H$ be a hole of length at least $6$ in $L_r$ and let $x_1\dd h_1\dd h_2\dd x_2\dd x_1$ be an $H$-flap in $G$ where $h_1,h_2\in H$ and $x_1,x_2\in L_{r-1}\cup L_r$. Then
$x_1,x_2\in L_{r}\setminus H$. 
\end{lemma}

\begin{proof}
Suppose not. Assume that $x_1,x_2\in L_r$. By symmetry, we may assume that $x_1\in L_r\setminus H$ and $x_2\in H$. But now $h_1\dd h_2\dd x_2$ is an $x_1$-sector of length $2$ in $H$, contrary to \Cref{lem:frombottom}. Therefore, one of $x_1,x_2$ belongs to $L_{r-1}$; say $x_1\in L_{r-1}$. Let $h_0$ be the unique neighbor of $h_1$ in $H$ other than $h_2$ and let $h_3$ be the unique neighbor of $h_2$ in $H$ other than $h_1$. Since $H$ has length at least $6$, it follows that $P=h_0\dd h_1\dd h_2\dd h_3$ is a path of length $3$ in $H$. By \Cref{lem:fromtop} applied to $H, x_1$, we have $N_P(x_1)=\{h_1\}$; in particular, $x_2\neq h_3$ and so $x_2\notin H$. By \Cref{lem:fromtop} and \Cref{lem:frombottom} applied to $H, x_2$ (depending on whether $x_2\in L_{r-1}$ or $x_2\in L_r\setminus H$), we have $N_P(x_2)=\{h_2\}$. Also, since $r\geq 2$, $x_1$ has a neighbor $y\in L_{r-2}$, and since $G$ is triangle-free, $x_2y\notin E(G)$. But now $(x_1,h_1,h_2,x_2;y,h_0,h_3)$ is a $4$-sunspot in $G$, a contradiction. This completes the proof of \Cref{lem:topflap}.
\end{proof}

Let us now prove \Cref{thm:flap}:

\begin{proof}[Proof of \Cref{thm:flap}]
By \Cref{lem:deepleveling}, there is an $r$-leveling $(L_0,\ldots, L_r)$ in $G$ for some integer $r\geq 3$ such that $\chi(L_r)=c+1$ and $L_r$ is both non-degenerate and liberal. We claim that $L=L_r$ is the desired induced subgraph of $G$. To see this, it suffices to show that $L_r$ is flapless. Suppose for a contradiction that there is a hole $H$ in $L_r$ of length at least $6$ for which there is an $H$-flap in $L_r$. By the definition, this means there is a $4$-hole $x_1\dd h_1\dd h_2\dd x_2\dd x_1$ in $L_r$ such that $h_1,h_2\in H$.

By \Cref{lem:topflap}, we have $x_1,x_2\in L_r\setminus H$. Let $h_0$ be the unique neighbor of $h_1$ in $H$ other than $h_2$ and let $h_3$ be the unique neighbor of $h_2$ in $H$ other than $h_1$. Since $H$ has length at least $6$, it follows that $P=h_0\dd h_1\dd h_2\dd h_3$ is a path of length $3$ in $H$. By \Cref{lem:frombottom} applied to $H, x_1$ and $x_2$, we have $N_P(x_i)=\{h_i\}$ for $i\in \{1,2\}$. Since $r\geq 3$, it follows that $h_1$ has a neighbor $y_1\in L_{r-1}$ and $x_2$ has a neighbor $y_2\in L_{r-1}$. 

By \Cref{lem:fromtop} applied to $H, y_1$, we have $N_{P}(y_1)=\{h_1\}$. Since $G$ is triangle-free, it follows that $x_1y_1\notin E(G)$. Also, by \Cref{lem:topflap}, $y_1\dd h_1\dd h_2\dd x_2\dd y_1$ is not an $H$-flap in $G$. Consequently, we have $x_2y_1\notin E(G)$; in particular $y_1\neq y_2$. We deduce that $h_1$ is the only neighbor of $y_1$ in $P\cup \{x_1,x_2\}$. Since $G$ is triangle-free, it follows that $x_1y_2,h_2y_2\notin E(G)$. Also, by \Cref{lem:topflap}, neither $y_2\dd h_1\dd h_2\dd x_2\dd y_2$ nor $x_2\dd h_2\dd h_3\dd y_2\dd x_2$ is an $H$-flap in $G$. Thus, we have $h_1y_2,h_3y_2\notin E(G)$. We deduce that $x_2$ is the only neighbor of $y_2$ in $\{h_1,h_2,h_3,x_1,x_2\}$. On the other hand, by \Cref{lem:topflap}, $y_2\dd h_0\dd h_1\dd y_1\dd y_2$ is not an $H$-flap in $G$, and so there is a vertex $z\in \{h_0,y_1\}$ for which $y_2z\notin E(G)$. But now $(h_1,h_2,x_2,x_1;z,h_3,y_2)$ is a $4$-sunspot in $G$, a contradiction. This completes the proof of \Cref{thm:flap}.
\end{proof}

\section{Flares}\label{sec:flairs}

Let $G$ be a graph. We say that $X,Y\subseteq V(G)$ are \textit{anticomplete in $G$} if $X\cap Y=\varnothing$ and there is no edge in $G$ with an end in $X$ and an end in $Y$. If $X=\{x\}$, then we say that $x$ is \textit{anticomplete to $Y$ in $G$} to mean $X$ and $Y$ are anticomplete in $G$. 

Given a hole $H$ in $G$, an \textit{$H$-flare in $G$} is a map $\Phi:H\rightarrow 2^{G\setminus H}$ such that for every $h\in H$, we have $\Phi(h)\subseteq N_G(h)\setminus H$ with $|\Phi(h)|\leq 1$. In this case, we write $\Phi[h]=\Phi(h)\cup \{h\}$. The \textit{order} of $\Phi$ is the number of vertices $h\in H$ with $|\Phi(h)|=1$, and we say that $\Phi$ is \textit{full} if the order of $\Phi$ is $|H|$. For $d\in \poi$, we say that $\Phi$ is \textit{$d$-safe} if for all vertices $h,h'\in H$ that are at distance at most $d$ in $H$, the sets $\Phi(h)$ and $\Phi[h']$ are anticomplete in $G$.

In this section, we show that:

\begin{theorem}\label{thm:flare}
    Let $d\in \poi$ and let $G$ be a triangle-free $4$-sunspot-free graph of minimum degree at least $4d-1$ that is both liberal and flapless. Let $H$ be a hole of length at least $6$ in $G$. Then there is a full $d$-safe $H$-flare in $G$.
\end{theorem}

First, we need a lemma:

\begin{lemma}\label{lem:onecommon}
Let $G$ be a triangle-free $4$-sunspot-free graph that is both liberal and flapless. Let $H$ be a hole of length at least $6$ in $G$ and let $\Phi$ be an $H$-flare in $G$. Let $h,h'\in H$ be distinct and let $x\in \Phi[h]$. Then
\begin{itemize}
    \item if $h,h'$ are adjacent, then $x,h'$ have no common neighbor in $G\setminus H$; and
    \item if $h,h'$ are non-adjacent, then $x,h'$ have at most one common neighbor in $G\setminus H$.
\end{itemize} 
\end{lemma}

\begin{proof}
Suppose not. Assume that $h,h'$ are adjacent and $x,h'$ have a common neighbor $x'\in G\setminus H$. Since $\{h,h',x'\}$ is not a triangle in $G$, it follows that $x\neq h$. But now $x\dd h\dd h'\dd x'\dd x$ is an $H$-flap in $G$, a contradiction. So, we may assume that $h,h'$ are non-adjacent and $h',x$ have two common neighbors $z,z'\in G\setminus H$. Since $G$ is triangle-free, it follows that $x\dd z\dd h'\dd z'\dd x$ is a $4$-hole in $G$. We claim that:

\sta{\label{st:yy'} There are vertices $y,y'\in H$ such that $x,y,h',y'$ are all distinct, $xy,h'y'\in E(G)$, and $\{x,y\}$ and $\{h',y'\}$ are anticomplete in $G$.}

Since $H$ has length at least $6$, there are vertices $v,v'\in H$ such that $h,v,h',v'$ are all distinct, $hv,h'v'\in E(G)$, and $\{h,v\}$ and $\{h',v'\}$ are anticomplete in $G$. Thus, if $x=h$, then \eqref{st:yy'} follows by choosing $y=v$ and $y'=v'$. Assume that $x\neq h$; thus, $x\in \Phi(h)$. In this case, let $y=h$ and let $y'=v'$. Then $x,y,h',y'$ are all distinct (because $h,h',v'\in H$ are distinct, and $x\notin H$), $xy=xh\in E(G)$ (because $x\in \Phi(h)$), and $h'y'=h'v'\in E(G)$. Also, $xh'\notin E(G)$ (recall that $x\dd z\dd h'\dd z'\dd x$ is a $4$-hole in $G$), and $xy'\notin E(G)$ since otherwise $z\dd h'\dd y'\dd x\dd z$ is an $H$-flap in $G$. Moreover, $y=h$ is anticomplete to $\{h',y'\}=\{h',v'\}$ in $G$. Hence, $\{x,y\}$ and $\{h',y'\}$ are anticomplete in $G$. This proves \eqref{st:yy'}.
\medskip

Let $y,y'\in H$ be as in \eqref{st:yy'}. Since $G$ is triangle-free, it follows that $\{y,y'\}$ and $\{z,z'\}$ are anticomplete in $G$. Since $G$ is liberal and $z,z'$ are distinct and non-adjacent, we may choose a vertex $w\in N_G(z)\setminus N_G(z')$ and a vertex $w'\in N_G(z')\setminus N_G(z)$. Since $x,h'$ are common neighbors of $z,z'$, it follows that $\{x,h'\}\cap \{w,w'\}=\varnothing$. In fact, since $G$ is triangle-free, $\{x,h'\}$ and $\{w,w'\}$ are anticomplete in $G$, and $\{y,y'\}\cap \{w,w'\}=\varnothing$. Also, since $z\dd h'\dd y'\dd w\dd z$ is not an $H$-flap, it follows that $wy'\notin E(G)$, and since $z'\dd h'\dd y'\dd w'\dd z'$ is not an $H$-flap, it follows that $w'y'\notin E(G)$. Now, since $(z,h',z',x; w,y',w')$ is not a $4$-sunspot in $G$, we have $ww'\in E(G)$. Since $G$ is triangle-free, we have $\{wy,w'y\}\not\subseteq E(G)$; say $wy\notin E(G)$. But then $(x,z,h',z';y,w,y')$ is a $4$-sunspot in $G$, a contradiction. This completes the proof of \Cref{lem:onecommon}.
\end{proof}

\Cref{thm:flare} is now almost immediate:

\begin{proof}[Proof of \Cref{thm:flare}]
    Suppose not. Let $\Phi$ be a $d$-safe $H$-flare in $G$ of maximum order. Since $\Phi$ is not full, there exists $h\in H$ such that $\Phi(h)=\varnothing$. Let $D$ be the set of all vertices in $H\setminus \{h\}$ that are at distance at most $d$ from $h$ in $H$ and let $U=\bigcup_{u\in D}\Phi[u]$; thus, we have $|D|\leq 2d$ and $|U|\leq 4d$. Since $\Phi$ is $d$-safe, it follows that $N_U(h)=N_H(h)$. Also, by \Cref{lem:onecommon}, $h$ has at most $4(d-1)$ neighbors in $G\setminus (H\cup U)$ with a neighbor in $U$. Since the degree of $h$ in $G$ is at least $4d-1>4(d-1)+2$, it follows that $h$ has a neighbor $x$ in $G\setminus H$ which is anticomplete to $U$ in $G$. Define the $H$-flare $\Phi':H\rightarrow 2^{G\setminus H}$ as follows: let $\Phi'(h)=\{x\}$, and let $\Phi'(h')=\Phi(h')$ for all $h'\in H\setminus \{h\}$. Since $\Phi$ is $d$-safe and $x$ is anticomplete to $U$ in $G$, it follows that $\Phi'$ is also $d$-safe. But the order of $\Phi'$ is strictly larger than the order of $\Phi$, a contradiction. This completes the proof of \Cref{thm:flare}.
\end{proof}

\section{Part assembly}\label{sec:end}

Here we complete the proof of \Cref{thm:maintrianglefree}. We need a couple of results from the literature, beginning with the following:

\begin{theorem}[Chudnovsky, Scott, Seymour; see 1.4 in \cite{SSIII}]\label{thm:mainfromiii}
For all $k,\ell\in \poi$, there is a constant $c_{\ref{thm:mainfromiii}}=c_{\ref{thm:mainfromiii}}(k,\ell)\in \poi$ such that for every graph $G$ with $\omega(G)\leq k$ and $\chi(G)>c_{\ref{thm:mainfromiii}}$, there is a hole of length at least $\ell$ in $G$.
\end{theorem}

We also need two results from \cite{SSIV}. Let $G$ be a graph. For a vertex $v\in V(G)$, we denote by $N_G^2(v)$ the set of all vertices that are at distance exactly two from $v$ in $G$.

\begin{theorem}[Scott and Seymour; see 3.1 in \cite{SSIV} and its proof]\label{thm:fromii}
Let $G$ be a triangle-free graph such that there is no $6$-hole in $G$. Then we have $\chi(N_G^2(v))\leq 2$ for every $v\in V(G)$.
\end{theorem}

\begin{theorem}[Scott and Seymour; see 3.10 in \cite{SSIV}]\label{thm:fromiii}
Let $\kappa\in \poi$ and let $\ell\geq 4$ be an integer. Let $G$ be a graph with no hole of length more than $\ell$ such that $\chi(N_G(v)),\chi(N_G^2(v))\leq \kappa$ for every vertex $v$ of $G$. Then $\chi(G)\leq (2\ell-2)\kappa$.
\end{theorem}

Combining these three results yields the following:

\begin{corollary}\label{cor:trianglefreeno6hole}
    For every integer $\ell\geq 6$, there exists $c_{\ref{cor:trianglefreeno6hole}}=c_{\ref{cor:trianglefreeno6hole}}(\ell)\in \poi$ such that for every triangle-free $G$ with $\chi(G)>c_{\ref{cor:trianglefreeno6hole}}$, there is a hole of length at least $\ell$ in $G$. Moreover, $c_{\ref{cor:trianglefreeno6hole}}(6)=20$ works.
\end{corollary}

\begin{proof}
    The existence of $c_{\ref{cor:trianglefreeno6hole}}$ follows from \Cref{thm:mainfromiii} applied to $k=2$ and $\ell$. Also, by \Cref{thm:fromiii,thm:fromii}, we have 
    $$\chi(G)\leq 2(2\times 6-2)=20$$
    for every triangle-free graph $G$ with no hole of length at least $6$.
\end{proof}

We are now ready to prove our main result, which we restate:

\maintrianglefree*

\begin{proof}
    Let $\tau=c_{\ref{cor:trianglefreeno6hole}}(\ell)$ be as in \Cref{cor:trianglefreeno6hole}, where $c_{\ref{cor:trianglefreeno6hole}}(6)=20$. Let $$c_{\ref{thm:maintrianglefree}}=c_{\ref{thm:maintrianglefree}}(\ell)=2\max\{\tau,4\ell-1\}+1.$$
    Note, in particular, that $c_{\ref{thm:maintrianglefree}}(6)=47$. We will show that the above value of $c_{\ref{thm:maintrianglefree}}$ satisfies the theorem. 

    Suppose for a contradiction that there is a graph $G$ with $\chi(G)>c_{\ref{thm:maintrianglefree}}$ which is triangle-free, $4$-sunspot-free, and $t$-sun-free for all $t\geq \ell$. By \Cref{thm:flap}, $G$ has a non-degenerate, liberal, and flapless induced subgraph $L$ with $\chi(L)=\max\{\tau,4\ell-1\}+1$. Since $\chi(L)>\tau$, by \Cref{cor:trianglefreeno6hole}, there is a hole of length at least $\ell\geq 6$ in $L$. Also, since $L$ is non-degenerate and $\chi(L)>4\ell-1$, it follows that $L$ has minimum degree at least $4\ell-1$.

    Let $H$ be a shortest hole of length at least $\ell$ in $L$; say $H$ has length $t\geq \ell$. Since $L$ is both liberal and flapless, and the minimum degree of $L$ is at least $4\ell-1$, it follows from \Cref{thm:flare} that there is a full $\ell$-safe $H$-flare $\Phi$ in $L$. For every $h\in H$, let $\Phi(h)=\{x_h\}$. We further claim that:

    \sta{\label{st:anticompleaves} Let $h,h'\in H$ be distinct. Then $x_h$ is anticomplete to $\{h',x_{h'}\}$ in $L$.}

    Suppose not. Let $h,h'\in H$ be distinct such that $x_h$ is not anticomplete to $\{h',x_{h'}\}$ in $L$, and subject to this property, the distance in $H$ between $h,h'$ is minimum. Let $P_1$ be a shortest path in $H$ from $h$ to $h'$, and let $P_2$ be the other path in $H$ from $h$ to $h'$. Since $\Phi$ is $\ell$-safe, it follows that both $P_1$ and $P_2$ have length at least $\ell+1$. From the choice of $h,h'$, it follows that $x_h$ is anticomplete to $P_1\setminus \{h'\}$ and $x_{h'}$ is anticomplete to $P_1\setminus \{h\}$. Also, since $x_h$ is not anticomplete to $\{h',x_{h'}\}$, it follows that there is a path $Q$ in $L$ from $h$ to $h'$ such that the interior of $Q$ is contained in $\{x_h,x_{h'}\}$; in particular, $Q$ has length $2$ or $3$. Consequently, $H'=h\dd P_1\dd h'\dd Q\dd h$ is a hole of length at least $\ell+3$ in $L$. On the other hand, since $P_2$ has length at least $\ell+1\geq 7$ and $Q$ has length at most $3$, it follows that $P_2$ is strictly longer than $Q$. But then $H'$ is a hole of length at least $\ell$ in $L$ that is strictly shorter than $H$, a contradiction. This proves \eqref{st:anticompleaves}.
    \medskip

    From \eqref{st:anticompleaves}, it follows that the induced subgraph of $L$ with vertex set $H\cup \{x_h:h\in H\}$ is a $t$-sun, a contradiction because $t\geq \ell$. This completes the proof of \Cref{thm:maintrianglefree}.
\end{proof}

\section{Acknowledgment}

This work was partly done when the first author attended the 2025 Barbados Graph Theory Workshop at Bellairs Research Institute in Holetown, Barbados. We thank the organizers for the invitation and for providing an engaging work environment. We also thank Xinyue Fan and Sahab Hajebi for helpful discussions, and the anonymous referees for their careful reading of the paper.

\bibliographystyle{plain}
\bibliography{ref}

\end{document}